\documentclass{article}[12pt]
\usepackage{url, graphicx, xcolor, hyperref, geometry}
\usepackage{latexsym,amsbsy}
\usepackage{lipsum}
\usepackage{amsfonts}
\usepackage{graphicx}
\usepackage{epstopdf}
\usepackage{algorithmic}
\usepackage{tikz}
\usetikzlibrary{matrix}
\usetikzlibrary{shapes,patterns,arrows,snakes,decorations.shapes}
\usetikzlibrary{positioning,fit,calc}
\usetikzlibrary{plotmarks}
\usepackage{amssymb,amsmath,amsthm}
\usepackage{todonotes}
\usepackage{smartdiagram}
\usepackage{mathabx}
\usepackage{float}

\newtheorem{theorem}{Theorem}[section]

\newtheorem{remark}{Remark}[section]

\newtheorem{definition}{Definition}[section]

\hypersetup{
    colorlinks,
    linkcolor={blue},
    citecolor={blue},
    urlcolor={blue}
}

\begin{document}
 
\title{A closed-form multigrid smoothing factor for an additive Vanka-type smoother applied to the Poisson equation}

\author{Chen Greif\thanks{Department of Computer Science, The University of British Columbia, Vancouver, BC, V6T 1Z4, Canada. The work of first author  was supported in part by a Discovery Grant of the Natural Sciences and Engineering Research Council of Canada. \tt{greif@cs.ubc.ca},  \tt{yunhui.he@ubc.ca}. } 
\and Yunhui He\footnotemark[1]}

\maketitle
 
\begin{abstract}
We consider an additive Vanka-type smoother for the Poisson equation discretized by the standard finite difference centered scheme. 
Using local Fourier analysis, we derive analytical formulas for the optimal smoothing factors for  two types of smoothers,
called vertex-wise and element-wise Vanka smoothers, and  present  the corresponding stencils. Interestingly, in one dimension 
 the element-wise Vanka smoother  is equivalent to the scaled mass operator obtained from the linear finite element method, and in two dimensions the  
element-wise Vanka smoother is equivalent to the scaled mass operator discretized by bilinear finite element method plus a scaled identity operator.  Based on these discoveries,
the mass matrix obtained from finite element method can be used as an approximation to the inverse of  the Laplacian, and the resulting mass-based relaxation scheme features  small  smoothing factors in one, two, and three dimensions.  Advantages of the mass operator are that the operator is sparse and well conditioned, and the computational cost of the relaxation scheme is only one matrix-vector product; there is no need to compute the inverse of a matrix. These findings may help  better understand the efficiency of additive Vanka smoothers and develop fast solvers for numerical solutions of partial differential equations.
\end{abstract}

\section{Introduction}

Consider  the Poisson equation in one dimension (1D) or  two dimensions (2D):
\begin{equation}
 -\Delta u = f, 
 \label{eq:Poisson}
 \end{equation}
where $u=u(x)$ and $f=f(x)$ in 1D or $u=u(x,y)$ and $f=f(x,y)$ in 2D. The function $f$ is assumed to be sufficiently smooth so that finite-difference discretizations of $u$ provide effective approximations of the solution.
The differential operator $\Delta$ stands for the {\em Laplacian}: $\Delta  \equiv \frac{d^2}{dx^2}$ in 1D and $\Delta  \equiv \frac{\partial^2 }{\partial x^2}+\frac{\partial^2}{\partial y^2}$ in 2D. We assume a uniform mesh discretization with meshsize $h$. 

We apply the standard three- and  five-point finite difference discretizations for the Laplacian; the corresponding stencils are  given by
\begin{equation}\label{eq:Laplace-stencil-1d}
  A_h = \frac{1}{h^2} \begin{bmatrix} -1 & 2 & -1 \end{bmatrix}
\end{equation}
and 
\begin{equation}\label{eq:Laplace-stencil-2d}
A_h=\frac{1}{h^2}
  \begin{bmatrix}
               &           -1            &               \\
   -1 &        4         & -1 \\
               &   -1                     &
  \end{bmatrix},
\end{equation}
respectively.

Let us denote the corresponding linear system by
\begin{equation}
  A_h u_h = b_h.
 \label{eq:linear_system}
\end{equation}

The numerical solution of \eqref{eq:linear_system} is one of  the most extensively explored topics in the numerical linear algebra literature. The discrete Laplacian $A_h$ is a symmetric positive definite M-matrix,  its eigenvalues are explicitly known, and it is used as a primary benchmark problem for the development of fast solvers. When the problem is large, iterative solvers that allow a high level of parallelism are often preferred.

One of the most efficient methods for solving \eqref{eq:linear_system} is {\em multigrid} \cite{stuben1982multigrid,MR1156079}. The choice of  multigrid components, such as a relaxation scheme as a smoother, plays an important role in the design of fast methods.  

The choice of additive Vanka as relaxation scheme, which is suitable for parallel computing, has recently drawn considerable attention. Vanka-type smoothers have been applied to   the Navier–Stokes equations \cite{john2002higher, SManservisi_2006a,SPVanka_1986a},  the Poisson equation using continuous and discontinuous finite elements methods \cite{he2021local},  the Stokes equations with finite element methods \cite{farrell2021local,voronin2021low}, poroelasticity  equations \cite{adler2021monolithic} in monolithic multigrid, and other problems.  A restricted additive Vanka for Stokes using $Q_1-Q_1$ discretizations is presented in \cite{saberi2021restricted}, which shows its competitiveness with the multiplicative Vanka smoother.  Nonoverlapping block smoothing using different patches has been discussed in  \cite{claus2021nonoverlapping} for the Stokes equations discretized by the marker-and-cell scheme.  Muliplicative Vanka smoothers  in combination with multigrid methods are discussed  in \cite{MR4024766,de2021two,franco2018multigrid}. Vanka-type relaxation has been used in many contexts, and for more details, we refer the reader to \cite{JAdler_etal_2015b,john2002higher, VJohn_LTobiska_2000a}.
 
In the literature,  Vanka-type relaxation schemes demonstrate their high efficiency in a multigrid setting, but there seems no theoretical analysis for the convergence speed even for the simple Poisson equation. In this work we take steps towards closing this gap by considering the additive Vanka relaxation for the Poisson equation, and exploring stencils for the Vanka patches. We derive analytical optimal smoothing factors for two types of additive Vanka patches  used in \cite{he2021local} for hybridized and embedded discontinuous {G}alerkin methods.    Moreover, we also find the corresponding stencils for the Vanka operator, and show that they are closely related to the scaled mass matrix obtained from the finite element method.  Based on this discovery, we propose the mass-based relaxation scheme, which yields rapid convergence. This mass-based relaxation is very simple: the computation cost is only matrix-vector product  and there is no need to solve the subproblems needed in an additive Vanka setting. Another advantage of the mass matrices obtained from (bi)linear elements is sparsity.  

Solvers for the Poisson equation often form the first step for designing fast solvers for more complex problems, such as the Stokes equations, and Navier-Stokes equations.  We therefore believe that the findings in this work can give some hints in future for designing fast numerical methods for these complex problems.  

The remainder of this paper is organized as follows. In Section \ref{sec:Vanka-intro} we introduce  the two types of additive Vanka smoothers for the Poisson equation.  In Section \ref{sec:LFA-Vanka}, we present our theoretical analysis of optimal  
smoothing factors in one and two dimensions.  Based on our analysis we also propose a 
mass-based smoother for the three-dimensional problem, where the mass matrix is obtained from the trilinear finite element method.  In Section \ref{sec:LFA-two-grid},  we numerically validate our analytical observations and  present an LFA two-grid convergence factor. Finally, in Section \ref{sec:conclusion} we discuss our findings and draw some conclusions.

\section{Vanka-type smoother}\label{sec:Vanka-intro}

We are interested in exploring the structure of  an additive  Vanka-type smoother for solving the linear system~\eqref{eq:linear_system} using multigrid. In general, this type can be thought of as related to the family of block Jacobi smoothers,
which are suitable for parallel computation and are typically highly efficient within the context of multigrid smoothing.

Let the degrees of freedom (DoFs) of $u_h$ be the set $\Upsilon$ such that $\Upsilon=\bigcup_{i=1}^{N}\Upsilon_i$. $V_i$ is a restriction operator that maps the vector  $u_h$ onto the vector in $\Upsilon_i$. Define $$A_i = V_i A_h V^T_i.$$ Then, we update current approximation $u^j$ by a single Vanka relaxation given by: \\
For $i=1,\cdots,N$,
\begin{equation*}
      A_i \delta u_i=V_i(b_h-A_hu^j)
\end{equation*}
and
\begin{equation*}
    u^{j+1} = u^j +\sum_{i=1}^{N}V_i^{T}W_i \delta u_i.
\end{equation*}

A single iteration of Vanka can be represented as
\begin{equation}\label{ASM-precondition}
  M = \sum_{i=1}^{N}V_i^{T}W_i A_i^{-1}V_i,
\end{equation}
where  the weighting matrix $W=W_i$ is given by the natural weights of the overlapping block decomposition. Each diagonal entry is equal to the reciprocal of the number 
of patches that the corresponding degree of freedom appears in. We refer to $M$ the {\em Vanka operator}.

For a single additive Vanka relaxation process, the relaxation error operator is given by
\begin{equation}\label{eq:relxation-error-operator}
  S = I -\omega M A_h.
\end{equation}
A key factor is the choice of the patch, that is, the $\{ \Upsilon_i \}$. Here, following \cite{he2021local}, we consider two patches, shown in Figure \ref{fig-patch}. We refer to the left patch 
in Figure \ref{fig-patch} as an {\it{element-wise patch}} and the right one as a {\it{vertex-wise patch}}. We denote the corresponding  relaxation error operators in \eqref{eq:relxation-error-operator} 
as $S_{e}$ and $S_v$, respectively.  The collection of circles indicate the number of DoFs in one patch $\Upsilon_i$. This means that the resulting subproblem is associated with a small matrix $A_i$ whose size is $4\times 4$ or $5\times 5$.  In the remainder of this work,  for simplicity and clarity we use subscripts $e$ and $v$ to distinguish between the corresponding operators for element-wise Vanka and vertex-wise Vanka, respectively. 

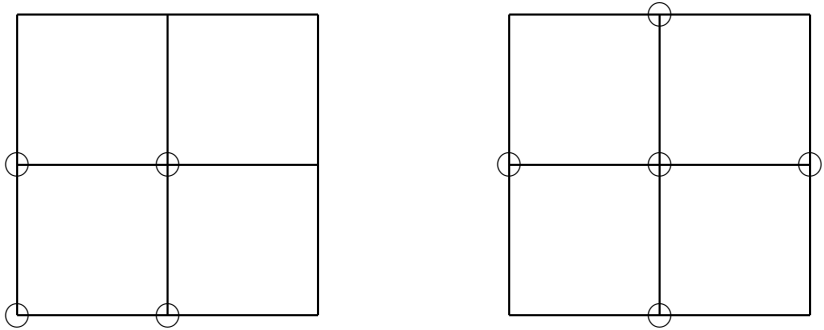
\begin{figure}[htp]
\centering

\begin{tikzpicture}[font=\normalsize]
  \draw [-] [thick] (0,0) --(0,4);
  \draw [-] [thick] (0,4) --(4,4);
  \draw [-] [thick] (0,0) --(4,0);
  \draw [-] [thick] (4,0) --(4,4);
  \draw [-] [thick] (0,2) --(4,2);
  \draw [-] [thick] (2,0) --(2,4);
  \node at (2,0) {$\bigcirc$};
  \node at (0,2) {$\bigcirc$};
  \node at (2,2) {$\bigcirc$};
  \node at (0,0) {$\bigcirc$};
\end{tikzpicture}
\hspace{2cm}
\begin{tikzpicture}[font=\normalsize]
  \draw [-] [thick] (0,0) --(0,4);
  \draw [-] [thick] (0,4) --(4,4);
  \draw [-] [thick] (0,0) --(4,0);
  \draw [-] [thick] (4,0) --(4,4);
  \draw [-] [thick] (0,2) --(4,2);
  \draw [-] [thick] (2,0) --(2,4);
  \node at (2,0) {$\bigcirc$};
  \node at (0,2) {$\bigcirc$};
  \node at (2,2) {$\bigcirc$};
  \node at (4,2) {$\bigcirc$};
  \node at (2,4) {$\bigcirc$};
\end{tikzpicture}
 \caption{Left: element-wise patch. Right: vertex-wise patch. }\label{fig-patch}
\end{figure}
%
\section{Local Fourier analysis}\label{sec:LFA-Vanka}

Local Fourier analysis (LFA) \cite{MR1807961,wienands2004practical} is a useful tool for predicting and analyzing the convergence behavior of multigrid  and other numerical algorithms. In this section, we employ LFA to examine the spectrum or spectral radius of the underlying operators to better understand  the proposed Vanka-type smoother.

When applying  LFA to multigrid methods, high and low frequencies \cite{MR1807961} for standard coarsening ($H=2h$) are introduced:
\begin{equation*}
  \boldsymbol{\theta}\in T^{{\rm low}} =\left[-\frac{\pi}{2},\frac{\pi}{2}\right)^{d}, \quad  \boldsymbol{\theta}\in T^{{\rm high}} =\displaystyle \left[-\frac{\pi}{2},\frac{3\pi}{2}\right)^{d} \bigg\backslash \left[-\frac{\pi}{2},\frac{\pi}{2}\right)^{d},
\end{equation*}
where $d$ is the dimension of the underlying problem. 

To quantitatively assess the performance of multigrid, LFA-predicted two-grid convergence and smoothing factors are used. In many cases the LFA smoothing factor is used due to its simplicity. It works under the assumption that the smoothing process 
reduces the high frequencies and leaves the low frequencies unchanged. The LFA smoothing factor typically offers a rather sharp bound on the actual two-grid performance. 
\begin{definition}\label{def:smoothing-factor}
Let $ S$ be the relaxation error operator. Then, the corresponding LFA smoothing factor for $\mathcal{S}$ is given by
\begin{equation}\label{eq:smoothing-factor}
  \mu(\omega)=\max_{\boldsymbol{\theta}\in T^{{
  \rm high}}}\Big\{\rho( \widetilde{S}(\boldsymbol{\theta},\omega)) \,\,\Big\},
\end{equation}
where $\widetilde{S}(\boldsymbol{\theta},\omega)$ is the symbol of $S$, $\omega$ is algorithmic parameter, and $\rho(\widetilde{S}(\boldsymbol{\theta},\omega))$ denotes the spectral radius of the matrix $\widetilde{S}(\boldsymbol{\theta},\omega)$.
\end{definition}
Note that the LFA smoothing factor $\mu$ is a function of $\omega$.  Often, one can minimize \eqref{eq:smoothing-factor} with respect to $\omega$ to obtain fast convergence speed. We define the optimal smoothing factor as
\begin{equation}
\mu_{\rm opt} = \min_{\omega} \mu.
\end{equation}
In this work, the symbol $\widetilde{S}$ for the Laplacian considered is a scalar, so the spectral radius is reduced to the maximum of a scalar function.  In the following, we use LFA to identify the optimal smoothing factor for the additive Vanka-type relaxation schemes and explore the structure of
the Vanka operator $M$ defined in \eqref{ASM-precondition}.  Before providing our detailed analysis of the smoothing factor for different relaxation schemes,
we summarize our results in  Table \ref{tab:smoothing-factor}. The table  provides a review of quantitative results  
of additive Vanka smoothers, and for comparison we include results for the standard point-wise damped Jacobi smoother. 

\begin{table}[H]
 \caption{The optimal LFA smoothing factor in 1D and 2D}
\medskip
\centering
\begin{tabular}{c c  c c }
\hline
   Smoother   &Jacobi       &AS-e               &AS-v     \\
\hline 
 \multicolumn{4} {c} {1D} \\
$\omega_{\rm opt}$                  & 2/3           &12/17        &  81/104                   \\
$\mu_{\rm opt}$                     &0.333         &0.059       &0.039                   \\ \hline
 \multicolumn{4} {c} {2D} \\
$\omega_{\rm opt}$                 &4/5          &24/25       &20/23                \\
$\mu_{\rm opt}$                    &0.600           &0.280        &0.391            \\
\hline
\end{tabular}\label{tab:smoothing-factor}
\end{table}

Note that   a general form of the symbol of additive Vanka operator for the Stokes equations has been discussed in \cite{farrell2021local},  which gives
$\widetilde{M} = \widetilde{V}^T \widetilde{W} \Phi^T A_i^{-1} \Phi \widetilde{V}$, where  $\Phi$  is called the {\em relative Fourier matrix}. Here, we can directly
apply the formula of $\widetilde{M} $ to our additive Vanka operator; see \cite{farrell2021local}.  

In the analysis that follows, we will use $\iota$ to denote the imaginary scalar satisfying
$$ \iota^2=-1.$$

\subsection{Symbols of Vanka patches in 1D}
In this subsection, we first consider the analytical symbol of the element-wise patch, then the vertex-wise patch for the Laplacian in 1D. 
 We discuss the optimal smoothing factor for each case and derive
the corresponding stencil for  the Vanka operator.   

\subsubsection{Element-wise Vanka patch  in 1D}
 It can easily be shown that  the symbol of $A_h$, see \eqref{eq:Laplace-stencil-1d},  is given by
\begin{equation}\label{eq:symbol-A-1d}
\widetilde{A}_h =\frac{1}{h^2}2(1-\cos\theta).
\end{equation}
Moreover, for the element-wise patch the subproblem matrix is
\begin{equation*}
   A_i = \frac{1}{h^2} 
         \begin{pmatrix}
2 & -1\\
-1 & 2
          \end{pmatrix}.
\end{equation*}
Following \cite{farrell2021local}, the  relative Fourier matrix $\Phi$  is    
\begin{equation*}       
 \Phi =\begin{pmatrix}
   1 & 0\\
 0&  e^{\iota \theta}
\end{pmatrix} .
\end{equation*}
Then,  the symbol of $M_{e}$ is given by $\widetilde{M}_{e} = \widetilde{V}^T \widetilde{W} \Phi^T A_i^{-1} \Phi \widetilde{V}$, where
\begin{equation*}
  \widetilde{V} =  
         \begin{pmatrix}
1 \\
1 
          \end{pmatrix}, \quad 
  \widetilde{W} =\frac{1}{2}\begin{pmatrix}
   1 & 0\\
 0&  1
\end{pmatrix} ,\quad
A_i^{-1} =\frac{h^2}{3} \begin{pmatrix}
 2 & 1\\
1 & 2
\end{pmatrix}.
\end{equation*}
Based on the above formulas, we obtain
\begin{equation}\label{eq:element-Vanka-symbol-1d}
\widetilde{M}_{e} =\frac{h^2}{6} (4+e^{\iota \theta}+ e^{-\iota \theta}).
\end{equation}
Formula \eqref{eq:element-Vanka-symbol-1d}  indicates that the element-wise Vanka patch corresponds to the stencil
\begin{equation}\label{Vanka-e-stencil-1d}
M_{e} = \frac{h^2}{6}  \begin{bmatrix}
 1 & 4 & 1
\end{bmatrix}.
\end{equation}
 Recall that the mass stencil  in 1D using linear finite elements is given by
\begin{equation}
M_{fe} = \frac{h}{6}  \begin{bmatrix}
 1 & 4 & 1
\end{bmatrix}.
\end{equation}
This means that the element-wise Vanka operator is equivalent to a scaled mass matrix obtained from the linear finite element method,  $M_{e}=h M_{fe}$.

Next, we give the optimal smoothing factor for the element-wise Vanka relaxation scheme.
\begin{theorem}
The optimal smoothing factor of $S_{e}$ for the vertex-wise Vanka in 1D  is 
\begin{equation}\label{eq:smoothing-AS-e-1d}
  \mu_{e,\rm opt}= \min_{\omega} \max _{\theta \in T^{\rm high}}{|1- \omega \widetilde{ M} _{e} \widetilde{A}_h| }=\frac{1}{17}\approx 0.059,
\end{equation}
where the minimum is uniquely achieved  at $\omega=\omega_{\rm opt} =\frac{12}{17}\approx 0.706$.
\end{theorem}
\begin{proof}
When $\theta\in T^{\rm high} =[\frac{\pi}{2}, \frac{3\pi}{2}]$,
\begin{equation*}
\widetilde{ M} _{e} \widetilde{A}_ h=\frac{2}{3}(2-\cos \theta-\cos^2\theta) \in \left[\frac{4}{3},\frac{3}{2}\right].
\end{equation*}
Thus,
\begin{equation*}
\mu(\omega)= \max {|1- \omega \widetilde{ M} _{e} \widetilde{A}_h| }= \max\left \{ \left|1-\omega\frac{4}{3}\right|, \left|1-\omega\frac{3}{2}\right|\right \}.
\end{equation*}
 To minimize $\mu(\omega)$, we require
\begin{equation*}
\left|1-\omega\frac{4}{3}\right| = \left|1-\omega\frac{3}{2}\right|,
\end{equation*}
which gives   $\omega=\frac{2}{4/3+3/2}=\frac{12}{17}$. Then,  $\mu_{\rm opt}=1-\frac{2}{4/3+3/2}\frac{4}{3}=\frac{1}{17}$.
\end{proof}
It is well known that the optimal smoothing factor for damped Jacobi relaxation for the Laplacian in 1D (with $\omega=\frac{2}{3}$) is $\frac{1}{3}\approx 0.333 \gg 0.059$.  This suggests that using  the additive Vanka smoother for multigrid  achieves much faster convergence.

%
 \subsubsection{Vertex-wise Vanka patches in 1D}
We now consider the vertex-wise patch.   The subproblem matrix is
\begin{equation*}
   A_i = \frac{1}{h^2} 
         \begin{pmatrix}
2 & -1 & 0\\
-1 & 2 & -1\\
0 & -1 & 2
          \end{pmatrix}.
\end{equation*}
Again, following \cite{farrell2021local}, the  relative Fourier matrix $\Phi$  is  
\begin{equation*}
  \Phi =\begin{pmatrix}
e^{-\iota \theta} & 0 & 0\\
   0& 1 & 0\\
 0& 0&   e^{\iota \theta}
\end{pmatrix} .
\end{equation*}
Then,  the symbol of $M_{v}$ is given by $\widetilde{M}_{v} = \widetilde{V}^T \widetilde{W} \Phi^T A_i^{-1} \Phi \widetilde{V}$, where
\begin{equation*}
  \widetilde{V} =  
         \begin{pmatrix}
1 \\
1 \\
1
          \end{pmatrix}, \quad 
  \widetilde{W} =\frac{1}{3}\begin{pmatrix}
   1 & 0 & 0\\
 0&  1 & 0\\
0& 0& 1
\end{pmatrix} ,\quad
A_i^{-1} =\frac{h^2}{4} \begin{pmatrix}
 3& 2 & 1\\
2 & 4 & 2\\
1& 2& 3
\end{pmatrix}.
\end{equation*}
Using the above formulas, we have
\begin{equation}\label{eq:symbol-Mv-1d}
\widetilde{M} =\widetilde{V}^T \widetilde{W} \Phi^T A_i^{-1} \Phi \widetilde{V} = \frac{h^2}{12} (10+ 4e^{\iota \theta}+ 4e^{-\iota \theta}+ e^{\iota 2\theta}+ 4e^{-\iota2 \theta}  ).
\end{equation}
Based on \eqref{eq:symbol-Mv-1d},  the  stencil of $M_v$ is 
\begin{equation}\label{Vanka-v-stencil-1d}
M_v = \frac{h^2}{12}  \begin{bmatrix}
 1 & 4 & 10& 4& 1
\end{bmatrix}.
\end{equation}
Compared with \eqref{Vanka-e-stencil-1d},  the vertex-wise Vanka uses a wider stencil.
\begin{theorem}
The optimal smoothing factor of $S_{v}$ for the vertex-wise Vanka in 1D  is  
\begin{equation}\label{eq:smoothing-AS-vertex-1d}
  \mu_{v,\rm opt} = \min_{\omega} \max_{\theta\in T^{\rm high}} {|1- \omega \widetilde{ M} _{v} \widetilde{A}_h|}=\frac{1}{26}\approx 0.039,
\end{equation}
where the minimum is uniquely achieved at $\omega=\omega_{\rm opt} =\frac{81}{104}\approx 0.779$. 
\end{theorem}
%
\begin{proof}
From \eqref{eq:symbol-A-1d} and \eqref{eq:symbol-Mv-1d}, we have
\begin{align*}
\widetilde{ M} _{v} \widetilde{A}_h& = \frac{1}{3}(1-\cos\theta)\left(5+4\cos \theta +\cos(2 \theta)\right),\\
                            & = \frac{2}{3}(1-\cos\theta)(2+2\cos \theta +\cos^2 \theta),\\
                            & = \frac{2}{3}(1-\cos\theta)\left( (\cos \theta +1)^2+1 \right).
\end{align*}
Note that when $\theta\in T^{\rm high} =[\frac{\pi}{2}, \frac{3\pi}{2}], \cos\theta \in[-1,0]$. Let $g(x)= \frac{2}{3}(1-x)\left( (x +1)^2+1 \right)$, where $x=\cos\theta \in[-1,0]$. 
To identify the range of $g(x)$, we first compute its derivative:
\begin{equation*}
g'(x) = \frac{-2 }{3}x(3x+2) <0,  \quad \text{for} \,\, x \in[-1,0].
\end{equation*}
It follows that 
\begin{equation*}
\left(\frac{10}{9}\right)^2=g(-2/3)\leq g(x)\leq g(0)=g(1)=\frac{4}{3}.
\end{equation*}
That is, $\widetilde{ M} _{v} \widetilde{A}_h \in [(\frac{10}{9})^2, \frac{4}{3}]$  for $\theta\in T^{\rm high}$ . To minimize $\mu(\omega)$, we require   $\omega=\frac{2}{4/3+(10/9)^2}=\frac{81}{104}$. Then,
 it follows that $\mu_{\rm opt} =1-\frac{2}{4/3+(10/9)^2}(\frac{10}{9})^2=\frac{1}{26}.$
 
\end{proof}
Again, the optimal smoothing factor for vertex-wise Vanka is significantly smaller (and hence better) than that of the damped Jacobi relaxation scheme.  
\subsection{Symbols of Vanka patches in 2D}
Similarly to previous subsection, we first consider the analytical symbol for the element-wise patch, then for the vertex-wise patch for the Laplacian in 2D. 
 
\subsubsection{Element-wise Vanka patch in 2D}
 The symbol of Laplace operator discretized by five-point stencil, see \eqref{eq:Laplace-stencil-2d}, is
\begin{equation}\label{eq:symbol-Ah}
\widetilde{A}_h =\frac{1}{h^2} (4-e^{\iota \theta_1} - e^{\iota \theta_2}-e^{-\iota \theta_1}-e^{-\iota \theta_2})=\frac{1}{h^2}\big(4-2\cos\theta_1-2\cos\theta_2\big).
\end{equation}
For the vertex-wise patch, following \cite{farrell2021local}, the  relative Fourier matrix $\Phi$  is  
\begin{equation*}
   \Phi =\begin{pmatrix}
   1&    0 &     0& 0 \\
   0&      e^{\iota \theta_1}&    0& 0 \\
   0& 0  & e^{\iota \theta_2}& 0\\
   0& 0 &  0 & e^{\iota (\theta_1+\theta_2)}
  \end{pmatrix}.
\end{equation*}
Then,  the symbol of $M_{e}$ is given by $\widetilde{M}_{e} = \widetilde{V}^T \widetilde{W} \Phi^T A_i^{-1} \Phi\widetilde{V} $, where
\begin{equation*}
  \widetilde{W} = \frac{1}{4}
\begin{pmatrix}
1 & 0 & 0& 0\\
0 & 1 & 0& 0\\
0& 0 & 1& 0\\
0 & 0 & 0& 1
\end{pmatrix} ,\quad 
 \widetilde{V}= \begin{pmatrix} 1 \\ 1 \\ 1 \\ 1\end{pmatrix}, \quad 
  A_i =\frac{1}{h^2}\begin{pmatrix}
  4 &  -1 &   -1  &  0  \\
  -1 &  4 &  0  &   -1 \\
  -1&   0 &  4&  -1 \\
  0 &  -1 & -1  & 4
  \end{pmatrix}.
\end{equation*}
We have
\begin{equation*}
  A_i^{-1} =h^2\begin{pmatrix}
       7/24 &          1/12   &        1/12     &      1/24\\
       1/12 &          7/24   &        1/24     &      1/12\\
       1/12 &         1/24    &       7/24      &     1/12\\
       1/24 &         1/12    &       1/12      &     7/24
  \end{pmatrix},
\end{equation*}
and it follows that  
\begin{align}
\widetilde{ M}_{e} & =\widetilde{V}^T \widetilde{W} \Phi^T A_i^{-1} \Phi \widetilde{V}, \nonumber\\
                    & =  \frac{h^2}{96}\big(28+4( e^{\iota \theta_1}+e^{-\iota \theta_1} +e^{\iota \theta_2}+e^{-\iota \theta_2})+e^{\iota (\theta_1+\theta_2)}
                   +e^{-\iota (\theta_1+\theta_2)}+(e^{\iota \theta_1}e^{-\iota \theta_2}+e^{-\iota \theta_1}e^{\iota \theta_2} )\big),\nonumber\\
                    & =\frac{h^2}{24} \big(7+2(\cos\theta_1 + \cos \theta_2)+ \cos \theta_1\cos\theta_2 \big). \label{eq:symbol-Me-2d}
\end{align}
Next, we give the optimal smoothing factor for the element-wise Vanka relaxation in 2D.
\begin{theorem}
 The optimal smoothing factor of $S_{e}$ for the element-wise Vanka relaxation in 2D is given by
\begin{equation}\label{eq:smoothing-AS-element}
  \mu_{e,\rm opt}= \min_{\omega} \max_{\boldsymbol{\theta}\in T^{{
  \rm high}}} {|1- \omega \widetilde{ M} _{e} \widetilde{A}_h|}=\frac{7}{25}\approx 0.280,
\end{equation}
where the minimum is uniquely achieved at $\omega=\omega_{\rm opt} =\frac{24}{25}\approx 0.960$. 
\end{theorem}
\begin{proof}
From \eqref{eq:symbol-Ah} and \eqref{eq:symbol-Me-2d}, we have 
\begin{align*}
\widetilde{ M}_{e}\widetilde{A}_h
  & =  \frac{1}{24} \big(7+2(\cos\theta_1 + \cos \theta_2)+ \cos \theta_1\cos \theta_2 \big)\big(4-2\cos\theta_1-2\cos\theta_2\big),\\
  & = \frac{1}{12} \big(7+2(\cos\theta_1 + \cos \theta_2)+ \cos \theta_1 \cos \theta_2 \big) \big(2-(\cos\theta_1+\cos\theta_2)\big),\\
  & = \frac{1}{12}\big(7+2(\eta_1 +\eta_2)+ \eta_1\eta_2\big) \big(2-\eta_1-\eta_2\big),
\end{align*}
where $\eta_1=\cos\theta_1, \eta_2=\cos\theta_2$.
Let
\begin{equation*}
  g(\eta_1,\eta_2) = \frac{1}{12}\big(7+2(\eta_1 +\eta_2)+ \eta_1\eta_2\big) \big(2-\eta_1-\eta_2\big).
\end{equation*}

We have $(\eta_1,\eta_2)\in[-1,-1]^2=:\Omega$ and $g(\eta_1,\eta_2)$ is continuous in $\Omega$. By the Extreme Value Theorem,
 $g$ achieves its extremal values at the boundary of $\Omega$ or its derivatives are zeros. We first consider the derivatives, 
\begin{align*}
g_{\eta_1} & = \frac{1}{12} (-3-4\eta_1-2\eta_1\eta_2-\eta_2^2-2\eta_2),\\
g_{\eta_2} & = \frac{1}{12} (-3-4\eta_2-2\eta_1\eta_2-\eta_2^2-2\eta_1).
\end{align*} 
 Solving $g_{\eta_1}=g_{\eta_1}=0$ gives $\eta_1=\eta_2=-1$.  Thus, $ g(-1,-1)=\frac{4}{3}$ is a possible global extreme value.

Next, we compute the extremal values of $g(\eta_1,\eta_2)$ at the boundary $\Omega$. Due to the symmetric of $g$, we only need to consider the following two cases.

\begin{itemize}
\item {\bf Case 1}: $\eta_1=-1$ and $\eta_2\in[-1,1]$.  We have
\begin{equation*}
  g(-1,\eta_2) = \frac{1}{12}(5+\eta_2)(3-\eta_2).
\end{equation*}
Thus,
\begin{equation*}
1 =g(-1,1) \leq g(-1,\eta_2)\leq g(-1,-1)=\frac{4}{3}.
\end{equation*}
\item {\bf Case 2}:   $\eta_1=1$ and $\eta_2\in[-1,1]$. We have
\begin{equation*}
  g(1,\eta_2) = \frac{1}{4}(3+\eta_2)(1-\eta_2).
\end{equation*}
Thus,
\begin{equation*}
0 =g(1,1) \leq g(1,\eta_2)\leq g(1,-1)=1.
\end{equation*}

\end{itemize}

If we restrict $(\theta_1,\theta_2)\in T^{\rm high}$,  we find that the maximum and minimum of $\widetilde{ M} _{e} \widetilde{A}_h=g(\eta_1,\eta_2) $  are given by
\begin{equation}
g(-1,-1) =\frac{4}{3}, \quad g(1,0) =\frac{3}{4},
\end{equation}
respectively. It follows that $\omega_{\rm opt} =\frac{2}{3/4+4/3}=\frac{24}{25} $ and  $\mu_{e,\rm opt}=  1-\frac{2}{4/3+3/4}\frac{3}{4}=\frac{7}{25}$. 
\end{proof}
It is well known that the optimal smoothing factor for damped Jacobi relaxation for the Laplacian in 2D is $\frac{3}{5}$ with $\omega=\frac{4}{5}$  \cite{MR1807961}. 
 This suggests that using the additive Vanka smoother for multigrid method, convergence is faster compared to the damped Jacobi relaxation scheme.

Based on the symbol of $M_e$, we obtain the stencil of $M_e$, 
\begin{equation}\label{eq:stencil-e-2d}
M_e = \frac{h^2}{96}
\begin{bmatrix}
 1& 4 & 1\\
4 & 28 & 4 \\
1 & 4 &1
\end{bmatrix}.
\end{equation} 
Recall that the mass matrix stencil using bilinear finite elements is 
\begin{equation}\label{eq:Mfe-2d}
M_{fe}= \frac{h^2}{36}
\begin{bmatrix}
 1& 4 & 1\\
4 & 16 & 4 \\
1 & 4 &1
\end{bmatrix}.
\end{equation} 
Now, we can make a connection between  $M_e$ and $M_{fe}$: 
\begin{equation}\label{eq:connection-Me-Mfe}
M_{e} = \frac{3}{8}M_{fe} + \frac{h^2}{8} \mathcal{I},
\end{equation}
where 
\begin{equation}
\mathcal{I} = \begin{bmatrix}
 0& 0 & 0\\
0 & 1 & 0 \\
0 & 0 &0
\end{bmatrix}.
\end{equation} 
The relationship \eqref{eq:connection-Me-Mfe} is interesting, and suggests that the mass matrix obtained from bilinear elements might be a good approximation to  the inverse of $A_h$. 
Let us, then, move to consider the mass matrix \eqref{eq:Mfe-2d} as an approximation to the inverse of $A_h$.
\begin{theorem}\label{thm:mass-precondition-smoothing-factor}
Given the mass stencil $M_{fe}$ in \eqref{eq:Mfe-2d} and the relaxation scheme  $S_{fe}=I-\omega M_{fe}A_h$ ,  the corresponding optimal smoothing factor is
\begin{equation}\label{eq:smoothing-AS-fe}
  \mu_{fe,\rm opt}= \min_{\omega} \max_{\boldsymbol{\theta}\in T^{{
  \rm high}}} {|1- \omega \widetilde{ M} _{fe} \widetilde{A}_h|}=\frac{1}{3}\approx 0.333,
\end{equation}
where the minimum is uniquely achieved at $\omega=\omega_{\rm opt} = \frac{3}{4}$. 
\end{theorem}
\begin{proof}
It can easily be shown that  $\widetilde{ M} _{fe} \widetilde{A}_h \in \left[\frac{8}{9}, \frac{16}{9}\right]$ for $(\theta_1,\theta_2) \in T^{\rm high}$.  Thus, the optimal $\omega$ is $\omega=\frac{2}{8/9+16/9}=\frac{3}{4}$. Then, $\mu_{fe, \rm opt}=1-\frac{2}{16/9+8/9}\frac{8}{9}=\frac{1}{3}$.
\end{proof}

From Theorem \ref{thm:mass-precondition-smoothing-factor}, we see that the optimal  smoothing factor of 0.333 for  the mass-based relaxation is close to the optimal smoothing factor 0.280 for the element-wise Vanka patch, and it is better than 0.391 obtained from the vertex-wise Vanka, see \eqref{eq:smoothing-AS-vertex}, discussed in the next subsection.  Thus,  mass matrix could be used as a good approximation to the inverse of the Laplacian considered here.

\subsubsection{Vertex-wise Vanka patches in 2D}
Now, we analyse the smoothing factor for the vertex-wise patch.  Following \cite{farrell2021local}, the  relative Fourier matrix $\Phi$  is  
\begin{equation}\label{eq:Phi-vertex-2d}
   \Phi =\begin{pmatrix}
   e^{-\iota \theta_2}&    0 &     0& 0 & 0\\
   0&      e^{-\iota \theta_1}&    0& 0 & 0\\
   0& 0 & 1& 0 & 0\\
   0& 0 & 0& e^{\iota \theta_1}& 0\\
   0& 0 & 0& 0 & e^{\iota \theta_2}
  \end{pmatrix}.
\end{equation}
The symbol of $M_{v}$ is given by $\widetilde{M}_{v} = \widetilde{V}^T \widetilde{W} \Phi^T A_i^{-1} \Phi \widetilde{V}$ with

\begin{equation}\label{eq:W-vertex-2d}
  \widetilde{W} = \frac{1}{5} I, \quad  \widetilde{V}= \begin{pmatrix} 1 \\ 1 \\ 1 \\ 1 \\1 \end{pmatrix},\quad
  A_i =\frac{1}{h^2}\begin{pmatrix}
  4 &  0&   -1 &   0  &  0  \\
  0 &  4&  -1  &   0 &   0 \\
  -1&   -1 &   4 &  -1 &   -1 \\
  0 & 0 & -1  &  4 &  0\\
  0 & 0 &  -1 &  0 &  4
  \end{pmatrix},
\end{equation}
where $I$ stands for identity matrix with size $5\times 5$.

It can be shown that
\begin{equation}\label{eq:inverse-A-vertex-wise}
    A_i^{-1} =h^2\begin{pmatrix}
       13/48 &       1/48  &         1/12   &        1/48  &         1/48\\
       1/48  &        13/48  &         1/12 &          1/48  &         1/48\\
       1/12  &        1/12   &        1/3   &         1/12   &        1/12\\
       1/48  &         1/48  &         1/12 &         13/48  &         1/48\\
       1/48  &         1/48  &        1/12  &         1/48  &        13/48
  \end{pmatrix}.
\end{equation}
From \eqref{eq:Phi-vertex-2d}, \eqref{eq:W-vertex-2d} and  \eqref{eq:inverse-A-vertex-wise}, we have  
\begin{align*}
 \widetilde{ M}_{v} & =\widetilde{V}^T \widetilde{W} \Phi^T A_i^{-1} \Phi \widetilde{V},\\
                     &=\frac{h^2}{240}\big(8(e^{-\iota \theta_2}+e^{\iota \theta_2}+e^{-\iota \theta_1}+e^{\iota \theta_1})+68+(e^{-\iota \theta_2}+e^{-\iota \theta_1})^2+(e^{\iota \theta_2}+e^{\iota \theta_1})^2 +2(e^{-\iota \theta_2}e^{\iota \theta_1}+e^{\iota \theta_2}e^{-\iota \theta_1})            \big),\\
                     & = \frac{h^2}{240}\big( 16(\cos\theta_1+\cos\theta_2)+68+ 2\big((\cos\theta_1+\cos\theta_2)^2-(\sin\theta_1+\sin\theta_2)^2\big) +4 \cos(\theta_1-\theta_2)\big)\\
                     & = \frac{h^2}{120} \big((\cos\theta_1+\cos\theta_2+4)^2+(\cos\theta_1+\cos\theta_2)^2+16\big).
\end{align*}
Based on the symbol of $M_v$, we can write down the corresponding stencil of $M_v$ as follows:
\begin{equation}\label{eq:stencil-v-2d}
M_v = \frac{h^2}{240}
\begin{bmatrix}
0 & 0 & 1 & 0 & 0 \\
0 & 2& 8 & 2 & 0 \\
1 & 8 & 68 & 8 & 1 \\
0 & 2& 8 & 2 & 0 \\
0 & 0 & 1 & 0 & 0 
\end{bmatrix}.
\end{equation} 
Now, we are able to give the optimal smoothing factor for the vertex-wise Vanka relaxation scheme.
\begin{theorem}
 The optimal smoothing factor of $S_{v}$ for the vertex-wise Vanka relaxation in 2D is 
\begin{equation}\label{eq:smoothing-AS-vertex}
  \mu_{\rm opt} = \min_{\omega} \max {|1- \omega \widetilde{ M} _{v} \widetilde{A}_h|}= \frac{9}{23}\approx 0.391,
\end{equation}
where the minimum is uniquely achieved at $\omega=\omega_{\rm opt} = \frac{20}{23}\approx 0.8696$.
\end{theorem}
\begin{proof}
We first compute 
 
\begin{align*}
\widetilde{ M}_{v}\widetilde{A}_h
  & =   \frac{1}{120} \big((\cos\theta_1+\cos\theta_2+4)^2+(\cos\theta_1+\cos\theta_2)^2+16\big)\big(4-2\cos\theta_1-2\cos\theta_2\big)\\
  & = \frac{1}{60}\big((4+\eta)^2+\eta^2+16 \big)(2-\eta),
  \end{align*}
where $\eta=(\cos\theta_1+\cos\theta_2) \in [-2,2]$ with $(\theta_1,\theta_2)\in \left[-\frac{\pi}{2},\frac{3\pi}{2}\right]^{2}$.

Let $g(\eta) = \frac{1}{60}\Big((4+\eta)^2+\eta^2+16 \Big)(2-\eta)$. We find that
\begin{equation*}
  g'(\eta) = \frac{1}{30}(-3\eta^2-4\eta-8)<0, \quad \forall \eta\in \mathbb{R}.
\end{equation*}
This means that $g(\eta)$ is  a decreasing function. Thus, for $\eta\in [-2,2]$, we have
\begin{equation*}
0 =g(\eta =2)\leq g(\eta) \leq g(\eta =-2)=\frac{8}{5}.
\end{equation*}
This means that  $g(\eta)\in[0,\frac{8}{5}]$.

If we restrict $(\theta_1,\theta_2) \in T^{\rm high}$, then $\eta \in[-2, 1]$. In this situation, we have
\begin{equation*}
\frac{7}{10} =g(\eta =1)\leq g(\eta) \leq g(\eta =-2)=\frac{8}{5}.
\end{equation*}
Since $ \widetilde{ M} _{v} \widetilde{A}_h \in [\frac{7}{10},\frac{8}{5}]$ for  $(\theta_1,\theta_2) \in T^{\rm high}$, we have 
 $\omega_{\rm opt} =\frac{2}{7/10+8/5}=\frac{20}{23}$ and $\mu_{\rm opt}=1-\frac{20}{23}\frac{7}{10}=\frac{9}{23}$.
\end{proof}
\begin{remark}
Note that in 2D the  element-wise Vanka, see \eqref{eq:stencil-e-2d}, uses fewer points than that of vertex-wise Vanka, see \eqref{eq:stencil-v-2d}. However,  the corresponding  optimal smoothing factor of element-wise Vanka, see \eqref{eq:smoothing-AS-element}, is smaller than that of vertex-wise Vanka, see \eqref{eq:smoothing-AS-vertex}, which is different than the case in 1D.
\end{remark}
\subsection{Extension to the 3D case}
While we do not include a smoothing analysis of Vanka-type solvers for the 3D case, we can still make a few interesting observations. In particular, motivated by our findings on the potential role of the mass matrix for relaxation,  we further explore the scaled mass matrix in 3D as an approximation to the inverse of the Laplacian. 
Let  
\begin{equation*}\label{eq:mass-3d}
M_3 = h^{-4}M_e \otimes M_e \otimes M_e,
\end{equation*}
where $M_e$ is defined in \eqref{Vanka-e-stencil-1d}. The symbol of $M_3$ can be obtained by tensor product given by
\begin{equation*}\label{eq:element-Vanka-symbol-3d}
\widetilde{M}_{3} =\frac{h^2}{27} (2+\cos \theta_1) (2+\cos \theta_2)  (2+\cos \theta_3) .
\end{equation*}
The symbol of $A_h$ in 3D is
\begin{equation*}\label{eq:symbol-A-3d}
\widetilde{A}_h =\frac{1}{h^2}2(3-\cos\theta_1-\cos\theta_2-\cos\theta_3).
\end{equation*}
Let $M_J=\frac{6}{h^2} I$,  which is the Jacobi matrix. 
\begin{theorem}\label{thm:Jacobi-smoothing-factor-3d}
If we consider the point-wise damped Jacobi as a preconditioner for the Laplacian, then  the corresponding optimal smoothing factor of $S=I-\omega M_{J}A_h$ is
\begin{equation}\label{eq:smoothing-Jacobi-3d}
  \mu_{J,\rm opt}= \min_{\omega} \max_{\boldsymbol{\theta}\in T^{{
  \rm high}}} {|1- \omega \widetilde{ M} _{J} \widetilde{A}_h|}=\frac{5}{7}\approx 0.714,
\end{equation}
where the minimum is uniquely achieved at $\omega_{\rm opt}=\frac{6}{7}\approx 0.857$.
\end{theorem}
\begin{proof}
Note that $\widetilde{ M} _{J} \widetilde{A}_h=\frac{1}{3}(3-\cos \theta_1 -\cos\theta_2-\cos \theta_3)$. For $\boldsymbol{\theta}\in T^{\rm high}$, $\widetilde{ M} _{J} \widetilde{A}_h \in [1/3, 2]$. Thus, to minimize $\mu(\omega)=|1-\omega \widetilde{ M} _{J} \widetilde{A}_h|$, we require  $\omega=\frac{2}{1/3+2}=\frac{6}{7}$.
 It follows $\mu_{J,\rm opt}=1-\frac{6}{7}\frac{1}{3}=\frac{5}{7}$.
\end{proof}
Next, we consider mass-based relaxation scheme for the Laplacian in 3D.
\begin{theorem}\label{thm:mass-precondition-smoothing-factor-3d}
Given the scaled mass stencil $M_{3}$ in \eqref{eq:mass-3d}  and mass-based relaxation scheme  $S_m=I-\omega M_{3}A_h$ in 3D,   the corresponding optimal smoothing factor  is
\begin{equation}\label{eq:smoothing-mass-3d}
  \mu_{m,\rm opt}= \min_{\omega} \max_{\boldsymbol{\theta}\in T^{{
  \rm high}}} {|1- \omega \widetilde{ M} _{3} \widetilde{A}_h|} = \frac{131}{212} \approx 0.618,
\end{equation}
where the minimum is uniquely achieved at $\omega=\omega_{\rm opt} =\frac{729}{848 }\approx 0.860 $. 
\end{theorem}
\begin{proof}
Let 
\begin{equation*}
T=\widetilde{ M} _{3} \widetilde{A}_h=\frac{2}{27}(2+\cos \theta_1) (2+\cos \theta_2)  (2+\cos \theta_2)(3-\cos\theta_1-\cos\theta_2-\cos\theta_3).
\end{equation*}
Define $g(x,y,z) =\frac{2}{27} (2+x)(2+y)(2+z)(3-x-y-z)$ with $x,y,z \in[-1,1]$. To find the extremal values of $g$, we will consider its derivatives and the function values at the boundary of underlying domain.   We compute the derivatives of $g$ with respect  to $x, y$ and $z$, given by
\begin{align*}
g_{x}&=\frac{2}{27} (2+y)(2+z)(1-2x-y-z),\\
g_{y} &=\frac{2}{27}(2+x)(2+z)(1-2y-x-z),\\
g_{z} &= \frac{2}{27}(2+x)(2+y)(1-2z-x-y).
\end{align*} 
Solving $g_{x}=g_{y}=g_{z}=0$ with $x,y,z\in[-1,1]$ gives $x=y=z=1/4$.  However, $\boldsymbol{\theta^*}=(\theta_1,\theta_2,\theta_3)$ such that $(\cos\theta_1,\cos\theta_2,\cos\theta_3)=(1/4,1/4,1/4)$ does not belong to $T^{\rm high}$. 

Let us define $\Omega_1=[-1,1]^3$, $\Omega_2=[0,1]^3$, and $\Omega=\Omega_1 \big\backslash \Omega_2$.  Note that $\Omega$ corresponds to $\boldsymbol{\theta}\in T^{\rm high}$.
To find the extremal values of $T$ for $\boldsymbol{\theta}\in T^{\rm high}$, we only need to find the extremal values of $T$ at the boundary of $\Omega$, denoted as $\partial \Omega$. 
Note that $\partial \Omega $  contains the following  four cases.\\
{\bf Case 1:}
\begin{align*}
 &x=-1, (y,z)\in[-1,1]^2,\\
 &y=-1, (x,z)\in[-1,1]^2,\\
 &z=-1, (x,y)\in[-1,1]^2.
\end{align*}
{\bf Case 2:}
\begin{align*}
 &x=1, (y,z) \in [-1,0]\times [-1,1],\\
 &y=1, (x,z) \in [-1,0]\times [-1,1],\\
 &z=1, (x,y) \in [-1,0]\times [-1,1].
 \end{align*}
{\bf Case 3:}
\begin{align*}
 &x=0, (y,z) \in [0,1]^2,\\
 &y=0, (x,z) \in [0,1]^2\\
 &z=0, (x,y) \in [0,1]^2.
 \end{align*}
{\bf Case 4:}
\begin{align*}
 &x=1, (y,z) \in [0,1]\times [-1,0],\\
 &y=1, (x,z) \in  [0,1]\times [-1,0],\\
 &z=1, (x,y) \in  [0,1]\times [-1,0].
 \end{align*}
Due to the symmetry of $g(x,y,z)$ and our interest of maximum and minimum of $g(x,y,z)$,  we only need to consider the following sets
\begin{align*}
\mathcal{D}_1&=\left\{ (x,y,z)| \, x=-1, (y,z)\in[-1,1]^2 \right \},\\
\mathcal{D}_2&=\left\{ (x,y,z)| \, x=1, (y,z) \in [-1,0]\times [-1,1] \right \},\\
\mathcal{D}_3&=\left\{(x,y,z)| \, x=0, (y,z) \in [0,1]^2\right \}, \\
\mathcal{D}_4&=\left\{(x,y,z)| \, x=1, (y,z) \in [0,1]\times [-1,0] \right \}.
\end{align*}
 For $\boldsymbol{\theta}\in T^{\rm high}$, we check the extremal values of $g(x,y,z)$ on the sets  $\mathcal{D}_1\bigcup\mathcal{D}_2 \bigcup \mathcal{D}_3\bigcup \mathcal{D}_4$, and find that the maximum of $T$ is $\frac{4\cdot 7^3}{3^6} $ achieved at $ (\cos \theta_1, \cos \theta_2,\cos \theta_3)=(0, 1/3,1/3)$ and the smallest value of $T$ is
 $\frac{4}{9}$ with $(\cos \theta_1, \cos \theta_2,\cos \theta_3)=(-1, -1, -1)$. Thus, the optimal  parameter is $\omega =\frac{2}{4/9+ 4\cdot 7^3/3^6}=\frac{729}{848 }\approx 0.860 $ and the corresponding smoothing factor is  $\mu_{m, \rm opt} =1 - \frac{729}{848} \frac{4}{9}=\frac{131}{212 } \approx 0.618$.
\end{proof}

\begin{remark}
One might consider a tensor-product generalization, i.e.,  $M_{v,3}=h^{-4} M_v \otimes M_v \otimes M_v$, where $M_v$ is defined in \eqref{Vanka-v-stencil-1d}, as an approximation to the inverse of the Laplacian in 3D. However,  with this choice, we find that the numerically optimal smoothing factor is 0.800, which is larger than 0.618, see \eqref{eq:smoothing-mass-3d}, obtained from the mass-based relaxation. Thus, we do not further explore $M_{v,3}$. 
\end{remark} 
 \section{Numerical experiments}\label{sec:LFA-two-grid}
 
In this section we compute the smoothing factor  to validate our theoretical results.  
We also report the LFA two-grid convergence factor and compare it to the corresponding  smoothing factor. 

In general, the two-grid  error-propagation operator can be expressed as \cite{stuben1982multigrid,MR1156079}
\begin{equation}
 E=S^{\nu_2} (I-P A_H^{-1}R A_h)S^{\nu_1},
 \end{equation} 
where $R$ is the restriction operator  from grid $h$ to grid $H$,  $P$ is the interpolation operator from  grid $H$ to grid $h$,  
$A_H$ represents  the coarse-grid operator, and the integers $\nu_1$ and $\nu_2$ are the numbers of  pre- and post-relaxation sweeps, respectively.   
 Here, $S$ is the relaxation error operator defined in \eqref{eq:relxation-error-operator}.
 
\begin{definition}\label{def:two-grid convergence-factor}
The LFA two-grid convergence factor \cite{MR1807961,wienands2004practical} for $E$ is defined as
\begin{equation}\label{eq:LFA two-grid convergence-factor}
  \rho=\max_{\boldsymbol{\theta}\in T^{\rm low}}\left\{\rho(\widetilde{E}(\boldsymbol{\theta},\omega))  \, \right\},
\end{equation}
where $\omega$ is a parameter, $ \widetilde{E} $ denotes the symbol of the two-grid error operator $E$, and 
 $\rho(\widetilde{E}(\boldsymbol{\theta},\omega))$ denotes the spectral radius of the matrix $\widetilde{E} $.
\end{definition}

For more details on how to obtain the  symbol of the two-grid error-propagation operator $E$, see \cite{MR1807961,wienands2004practical}.  In our tests, we  consider $P$ to be the standard (bi)linear interpolation and take $R=P^T$  and $A_H=RA_hP$. We take $\rho$  to be the LFA prediction sampled at 64 equispaced points in each dimension of the Fourier domain.
For simplicity, we denote by $T(\nu)$ the two-grid method with $\nu=\nu_1+\nu_2$.

In Figure \ref{fig:EV-TG-eigs} 
we present the eigenvalue distribution of  $\widetilde{ E}$  for the Vanka-type method  in 2D with the optimal value of $\omega$. We see that all eigenvalues of $\widetilde{ E}$  are real and their largest magnitude matches the optimal smoothing factor.  

 Figure \ref{fig:Element-wise-eigs} shows the magnitude of eigenvalues of $\widetilde{ E}_{e}$ and $\widetilde{ S}_{v}$  as a function of the Fourier modes, $(\theta_1,\theta_2)$. We see that the eigenvalues of the two-grid error operator are distributed almost evenly in the entire Fourier domain. On the other hand, for the vertex-wise Vanka, the largest magnitude
 occurs at the $(0,\pm \pi)$ or $(\pm \pi,0)$, see Figure \ref{fig:Vertex-wise-eigs}.  From the right of Figures \ref{fig:Element-wise-eigs} and  \ref{fig:Vertex-wise-eigs},  we see that the relaxation scheme reduces the high frequency errors. 

\begin{figure}[htb]
\includegraphics[width=0.49\textwidth]{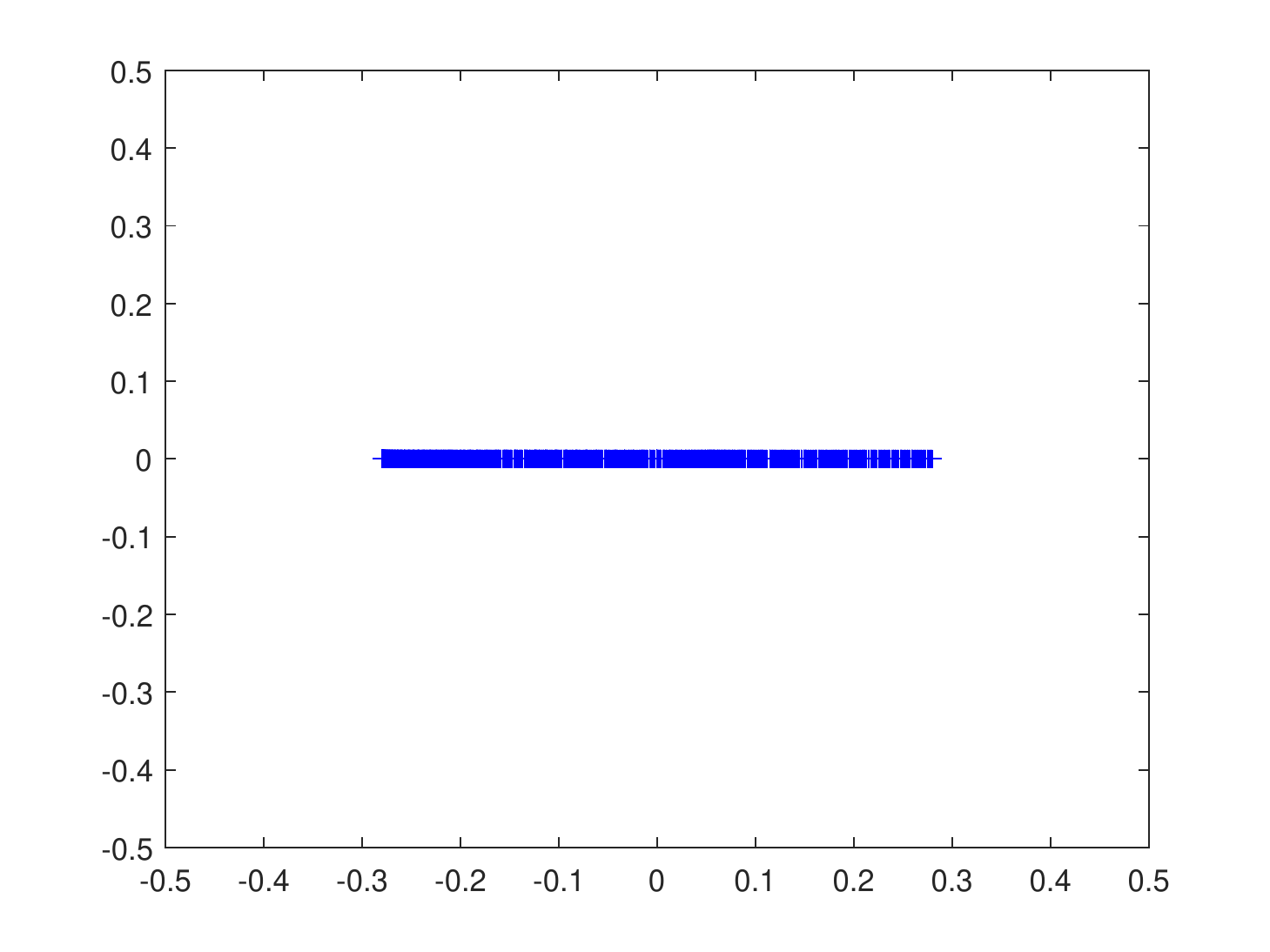}
\includegraphics[width=0.49\textwidth]{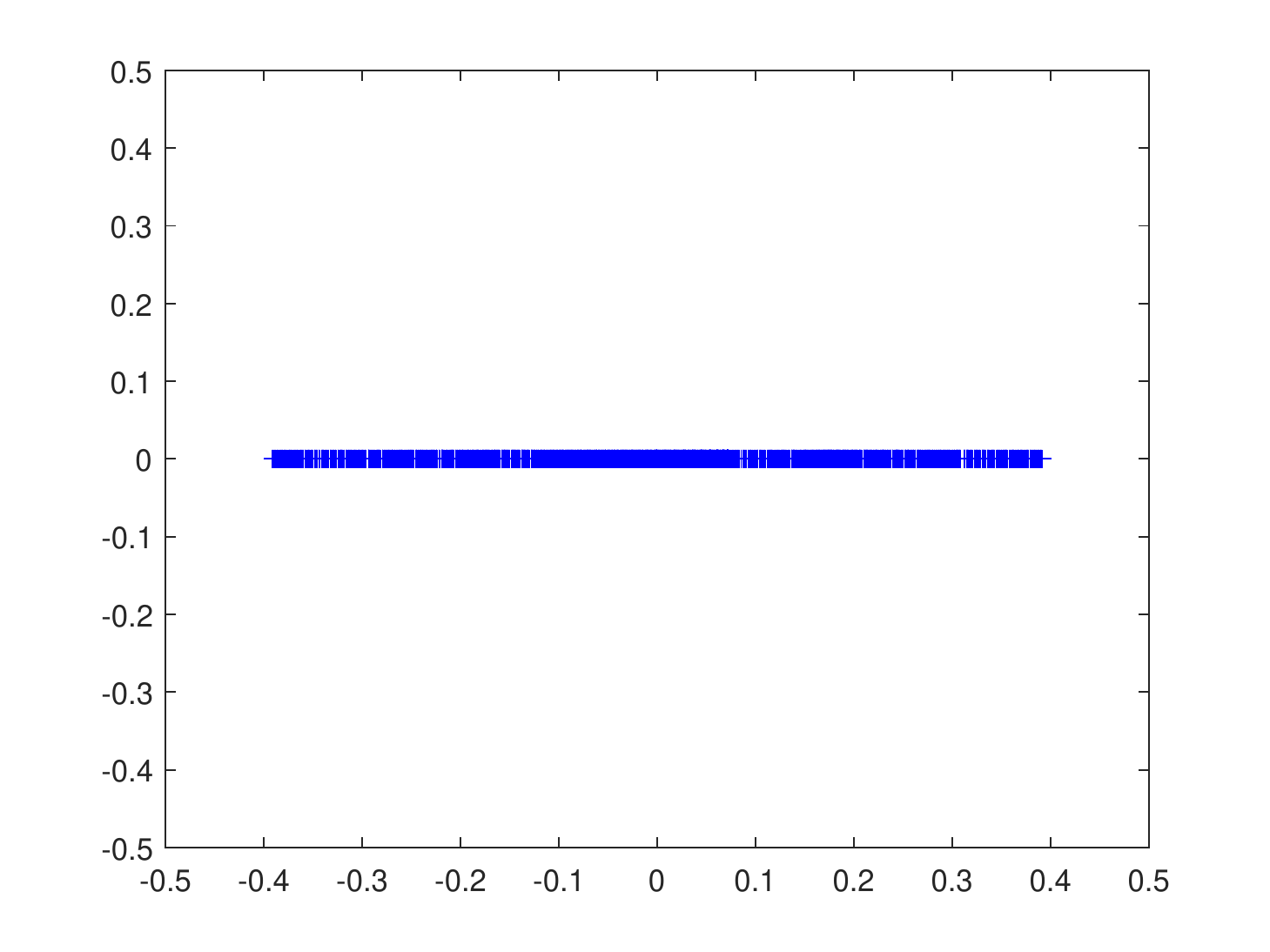}
  \caption{Eigenvalue distribution of  $\widetilde{ E}$  for Vanka-type method with $\nu=1$. Left: element-wise Vanka. Right: vertex-wise Vanka. }\label{fig:EV-TG-eigs}
\end{figure}

\begin{figure}[htb]
\includegraphics[width=0.49\textwidth]{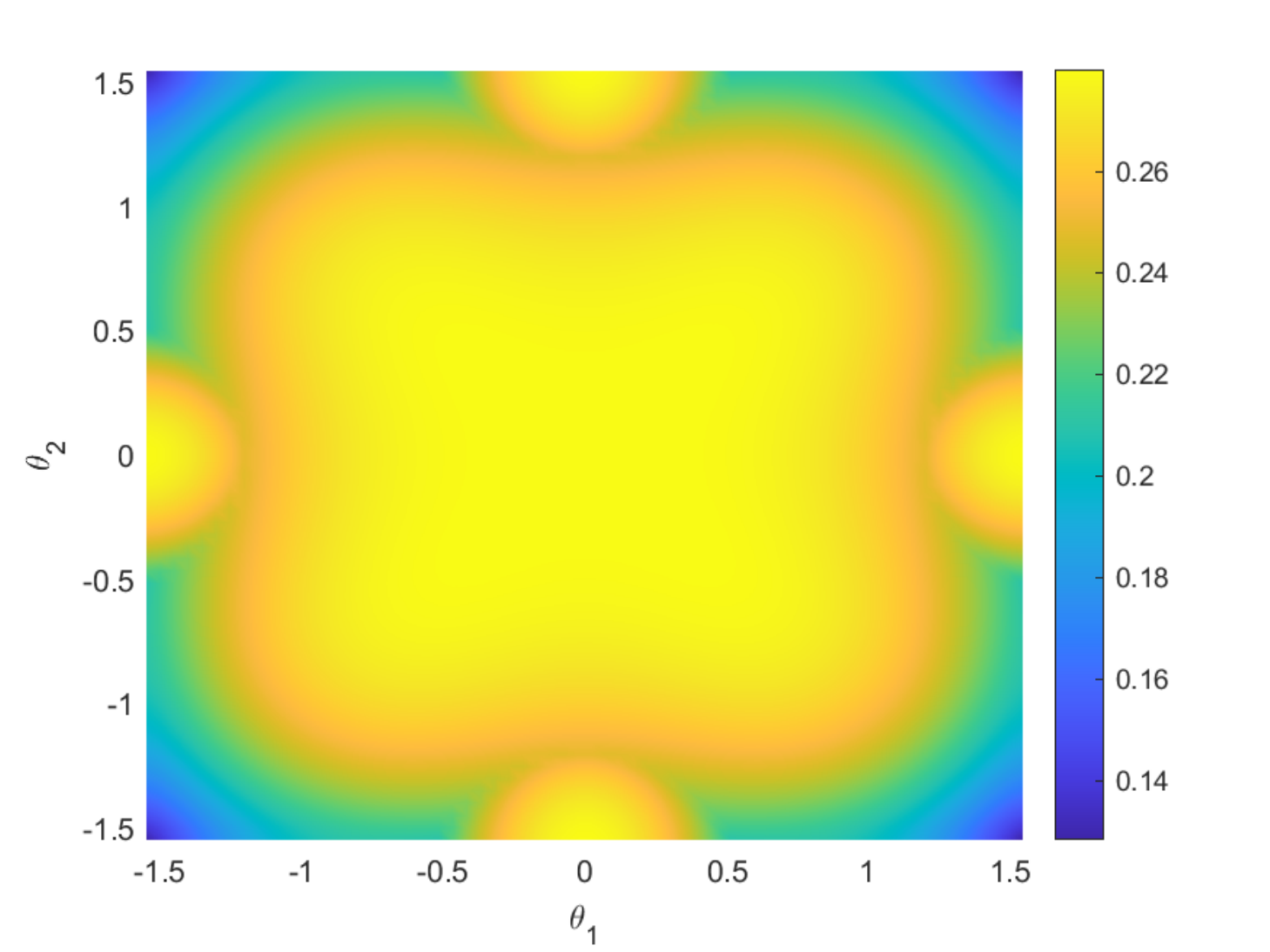}
\includegraphics[width=0.49\textwidth]{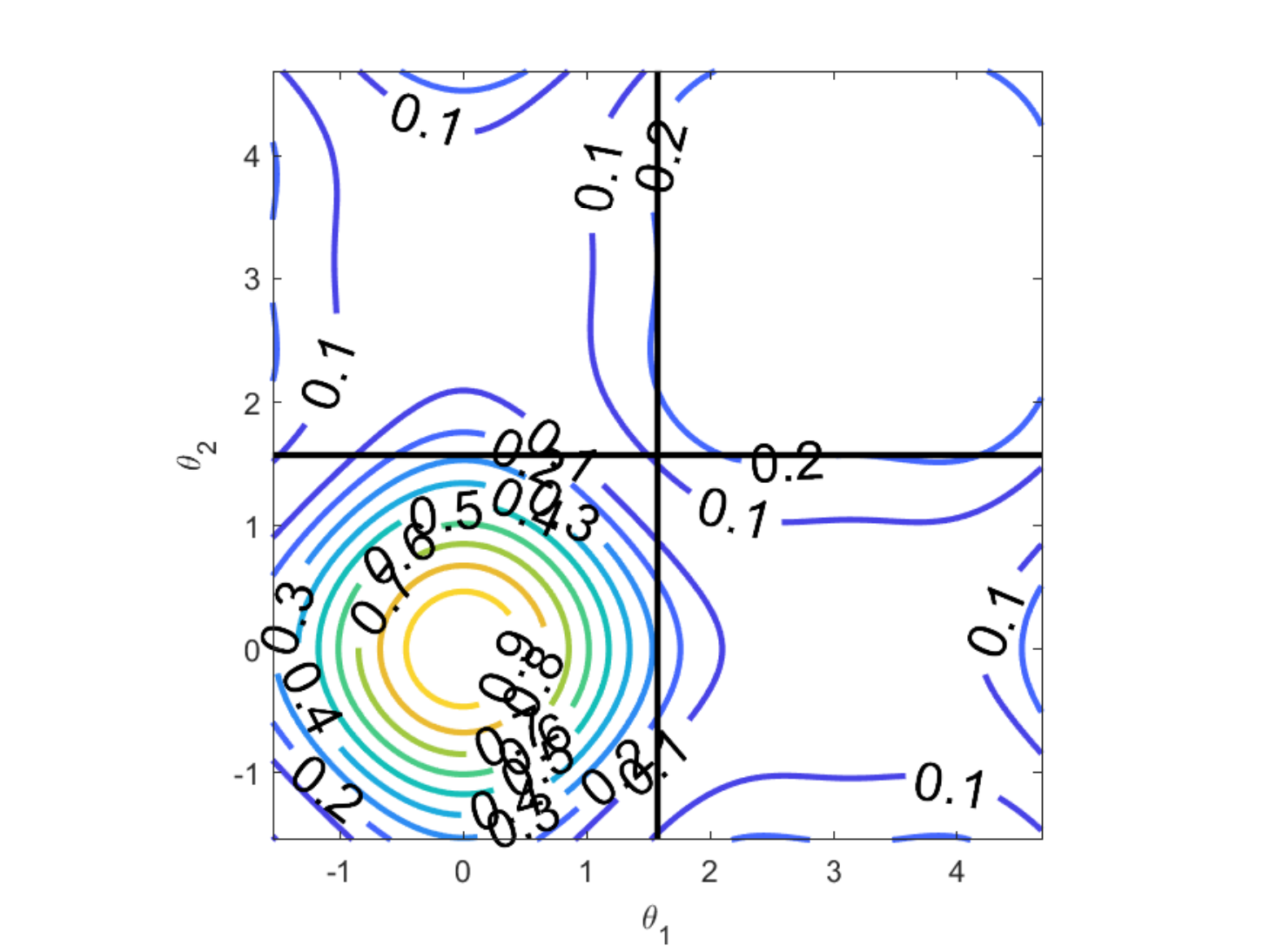}
  \caption{ Element-wise Vanka with $\nu=1$.  Left:  the magnitude of eigenvalues of $\widetilde{ E}_{e}$ as a function of the Fourier modes, $(\theta_1,\theta_2)$.  Right: the magnitude of eigenvalues of $\widetilde{ S}_{e}$ as a function of the Fourier modes, $(\theta_1,\theta_2)$.}\label{fig:Element-wise-eigs}
\end{figure}

\begin{figure}[H]
\includegraphics[width=0.49\textwidth]{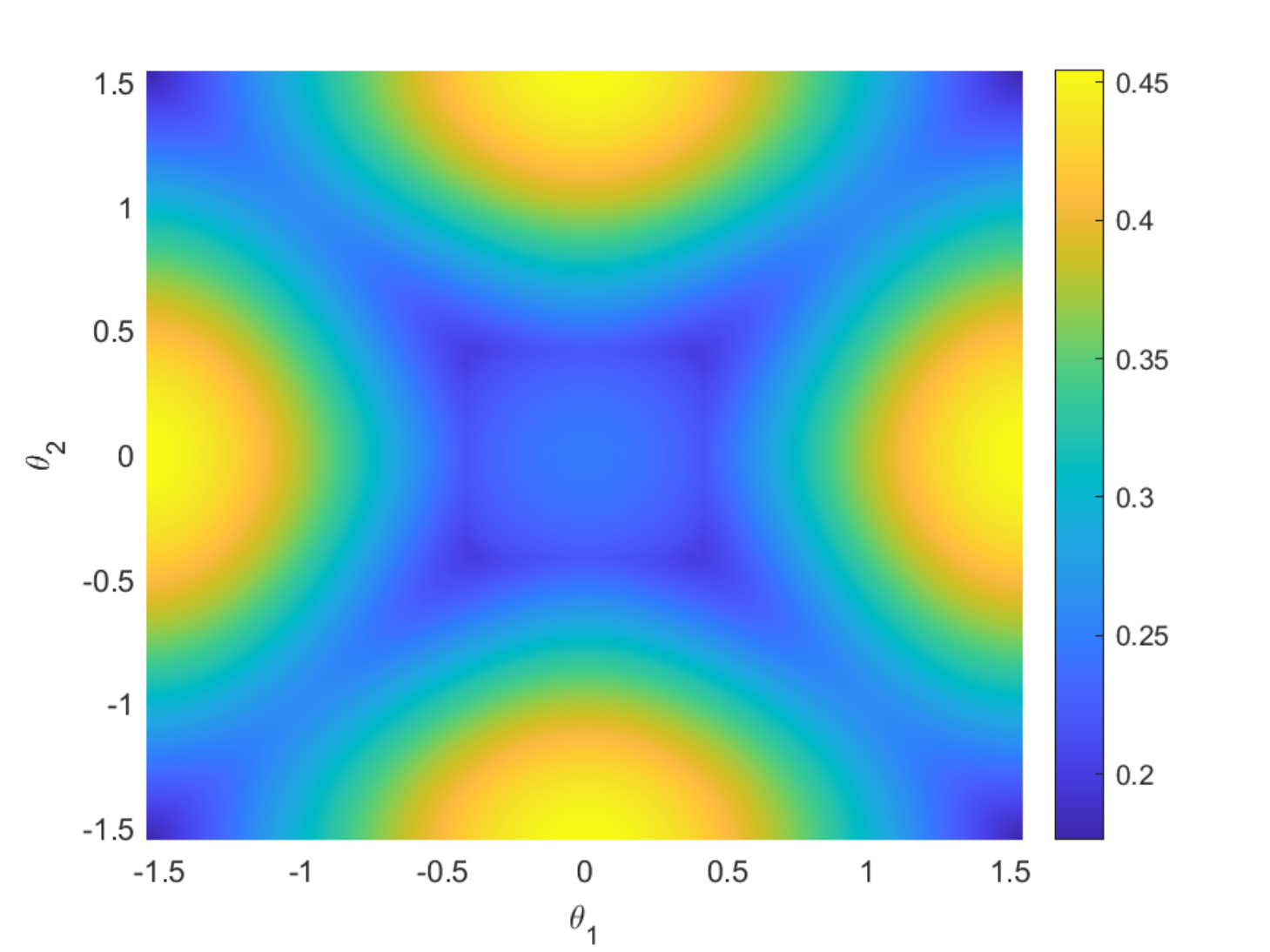}
\includegraphics[width=0.49\textwidth]{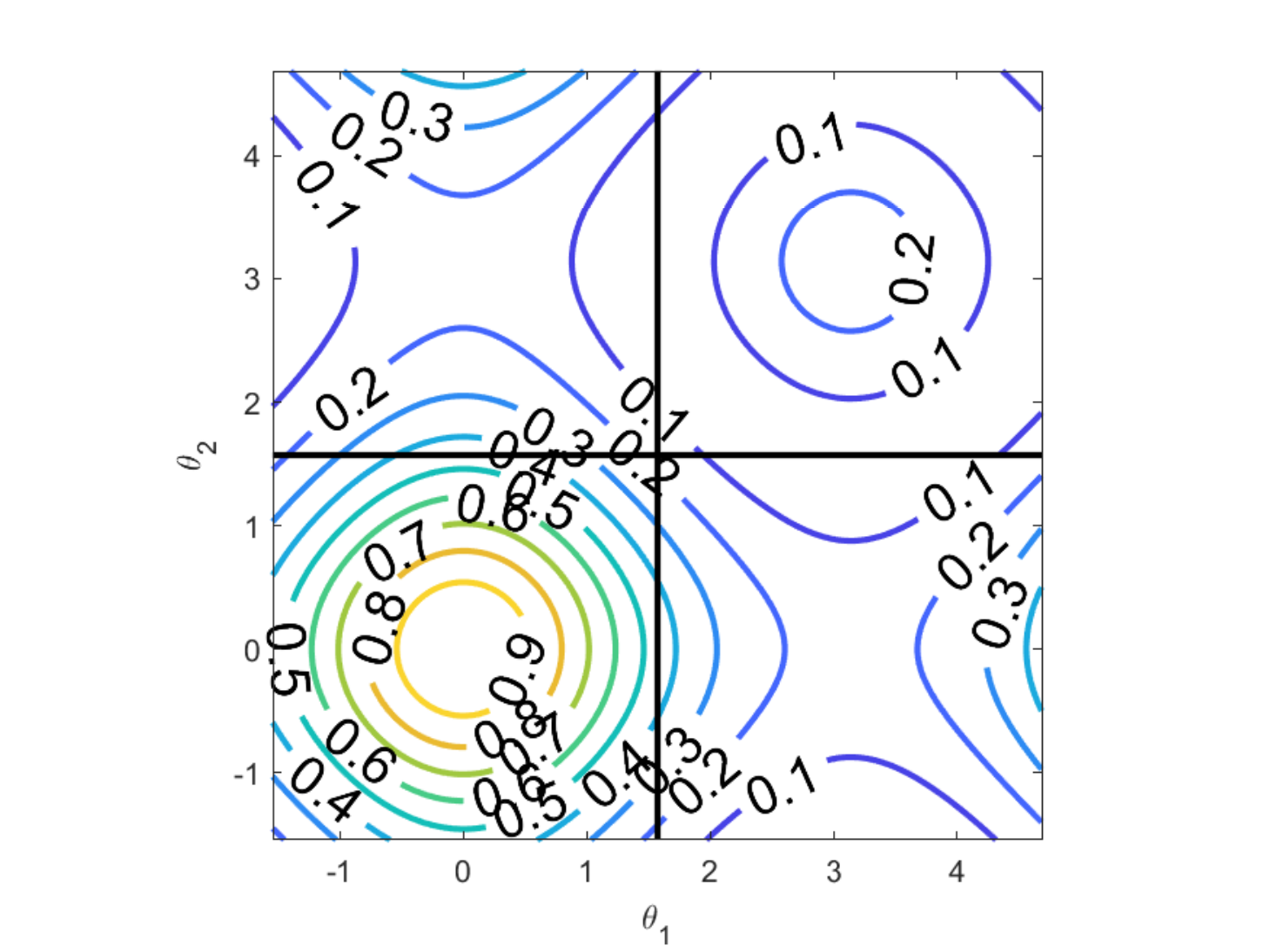}
  \caption{ Vertex-wise Vanka with $\nu=1$.  Left:  the magnitude of eigenvalues of $\widetilde{ E}_{v}$ as a function of the Fourier modes, $(\theta_1,\theta_2)$.  Right: the magnitude of eigenvalues of $\widetilde{ S}_{v}$ as a function of the Fourier modes, $(\theta_1,\theta_2)$.}\label{fig:Vertex-wise-eigs}
\end{figure}

\begin{table}[H]
 \caption{LFA predicted two-grid convergence factor, ${\rho}$ with  $\omega_{\rm opt}$  obtained from our analytical result and 
 $h=\frac{1}{64}$.}
\centering
\begin{tabular}{lcccccc}
\hline
Cycle                 &$\omega_{\rm opt}$  & $\mu_{\rm opt}$  &$TG(1)$  & $TG(2)$   &$TG(3)$   & $TG(4)$  \\ \hline
    \multicolumn{7} {c} { 1D} \\
$\rho_{ e}$           &12/17           &0.059    &0.059     &0.059      &0.040   &0.031   \\
$\rho_{ v}$          &81/104          &0.038    &0.091     &0.033      &0.022   &0.017   \\ \hline
    \multicolumn{7} {c} {2D} \\
$\rho_{e}$    &24/25          &0.280    &0.280     &0.092      &0.059    &0.045   \\
$\rho_{v}$    & 20/23         &0.391    &0.391     &0.153      &0.076   & 0.055  \\
$\rho_{fe}$    &3/4           &0.333    &0.333      & 0.111     &0.037   &0.029   \\ \hline
    \multicolumn{7} {c} {3D} \\
$\rho_{J}$      &6/7             &0.714     &0.714      &0.510      &0.364    &0.260   \\
$\rho_{m}$    & 729/848      &0.618     &0.618      & 0.382     &0.236    &0.146   \\ \hline
\end{tabular}\label{tab:mu-rho-results}
\end{table}

From Table \ref{tab:mu-rho-results}, we see that the smoothing factors match the two-grid convergence factor with $\nu=1$ (except for the vertex-wise patch for 1D), which is as expected since the smoothing 
factor typically offers a sharp prediction of the two-grid convergence factor.  As $\nu$ increases, we see a little degradation of $\rho(\nu)$, that is $\rho(\nu)>\mu^{\nu}$ for Vanka-type relaxation, and so does the mass-based relaxation scheme with $\nu=4$.

 To explore the  gap between the smoothing factor and two-grid convergence factor for the vertex-wise Vanka in 1D, we run our LFA code to minimize the LFA two-grid convergence factor with respect to  parameter $\omega$ by an exhaustive search with stepsize 0.02, and  find that  the optimal LFA two-grid convergence factor is  0.067 with $\omega=0.80$ and the corresponding smoothing factor is the same as the two-grid LFA convergence factor. For the choice of $\omega=\frac{81}{104}$ that minimizes the smoothing factor, the gap between the LFA two-grid convergence factor and the smoothing factor is reasonable, as is noted in the literature; see, for  example \cite{he2020two}. The smoothing factor is a sharp prediction of the two-grid convergence factor. 
\begin{remark}
The LFA two-grid convergence and smoothing factors seem independent of the meshsize $h$. We have tested different values of $h$ for Table \ref{tab:mu-rho-results} to confirm this. Further details are omitted.
\end{remark}

%
\section{Conclusions} \label{sec:conclusion}
We have presented a theoretical analysis of the optimal multigrid smoothing factor for two types of additive Vanka smoothers, applied to the Poisson equation discretized by the standard centered finite difference scheme.  The smoothers are shown to be highly efficient. We have found that the element-wise Vanka is closely related to the mass matrix obtained from the (bi)linear finite element method.  This observation  propels us to use the mass matrix as an approximation to the inverse of the Laplacian, and this yields rapid convergence. It is an interesting connection between the finite element method and finite difference methods.  

Solvers for equations related to the Laplacian are a natural first step in developing new algorithms for more complex problems, such as the Stokes equations, Navier-Stokes equations and other saddle-point problems. Using a similar approach as part of the development of fast solvers for those problems may prove computationally beneficial.

\bibliographystyle{abbrv}
\bibliography{mass_vanka_bib}
\end{document}